\documentclass[12pt,a4paper]{amsart}
\usepackage{mathrsfs}
\usepackage{amssymb,amsmath,amsthm,color}
\usepackage{graphicx,mcite}
\usepackage{hyperref}
\usepackage{url}
\usepackage{setspace}
\usepackage{enumerate}
\newtheorem{theorem}{Theorem}[section]
\newtheorem{lemma}[theorem]{Lemma}
\newtheorem{proof of lemma}[theorem]{Proof of Lemma}
\newtheorem{proposition}[theorem]{Proposition}

\newtheorem{corollary}[theorem]{Corollary}

\theoremstyle{definition}
\newtheorem{definition}[theorem]{Definition}

\newtheorem{remark}[theorem]{Remark}

\numberwithin{equation}{section}

\begin{document}

\title[Heisenberg uniqueness pair]
{Heisenberg uniqueness pairs for some algebraic curves and surfaces}

\author{D.K. Giri and R.K. Srivastava}

\address{Department of Mathematics, Indian Institute of Technology, Guwahati, India 781039.}
\email{deb.giri@iitg.ac.in, rksri@iitg.ac.in}


\date{\today}


\begin{abstract}

Let $X(\Gamma)$ be the space of all finite Borel measure $\mu$ in $\mathbb R^2$
which is supported on the curve $\Gamma$ and absolutely continuous with respect
to the arc length of $\Gamma$. For $\Lambda\subset\mathbb R^2,$ the pair
$\left(\Gamma, \Lambda\right)$ is called a Heisenberg uniqueness pair for $X(\Gamma)$
if any $\mu\in X(\Gamma)$ satisfies $\hat\mu\vert_\Lambda=0,$ implies $\mu=0.$ We
explore the Heisenberg uniqueness pairs corresponding to the cross, exponential
curves, and surfaces. Then, we prove a characterization of the Heisenberg uniqueness
pairs corresponding to finitely many parallel lines. We observe that the size of the
determining sets $\Lambda$ for $X(\Gamma)$ depends on the number of lines and their
irregular distribution that further relates to a phenomenon of interlacing of certain
trigonometric polynomials.

\end{abstract}

\maketitle

\section{Introduction}\label{section1}
The notion of the Heisenberg uniqueness pair has been first introduced by Hedenmalm
and Montes-Rodr\'iguez as a version of the uncertainty principle (see \cite{HR}).
We would like to mention that Heisenberg uniqueness pair up to some extent
is similar to an annihilating pair of Borel measurable sets of positive measure
as described by Havin and J\"{o}ricke (see \cite{HJ}). To describe this, consider
a pair of Borel measurable sets $\Gamma,\Lambda\subseteq\mathbb R.$ Then
$(\Gamma,\Lambda)$ form a mutually annihilating pair if any $\varphi\in L^2(\Gamma)$
whose Fourier transform $\hat\varphi$ supported on $\Lambda,$ implies $\varphi=0.$
\smallskip

Heisenberg uniqueness pair has a close relation with the long-standing problem
of determining the {\em exponential types} for finite measure which is eventually
about exploring the density of the set $\{e^{i\lambda t}:~\lambda\in\Lambda\}$ in
$L^2(\mu).$ For more details, we refer to Poltoratski \cite{AP1, AP2}.
In particular, the question of HUP can be thought as a dual problem of {\em gap problem}
(see \cite{AP1}). Let $\mu$ be a finite Borel measure which is supported on
a closed set $\Gamma\subset\mathbb R$ and
\[G_\Gamma=\sup\left\{a>0:~\exists~\mu\neq0,\text{ supp }\mu\subset\Gamma\text{ satisfies }\hat\mu\vert_{[0, a]}=0\right\}.\]
Let $\Lambda\subset\mathbb R.$ Then $\left(\Gamma, \Lambda\right)$
would be a HUP as long as the determining set $\Lambda$ intersect $\left[0, G_\Gamma\right]$
at most on a set of measure zero.

\smallskip

In addition, the concept of determining the Heisenberg uniqueness pair for a
class of finite measures has also a significant similarity with the celebrated
result due to M. Benedicks (see \cite{B}). That is, support of a function
$f\in L^1(\mathbb R^n)$ and its Fourier transform $\hat f$ both cannot be of
finite measure simultaneously. Later, various analogues of the Benedicks theorem
have been investigated in different set ups, including the Heisenberg group and
Euclidean motion groups (see \cite{NR, PS, SST}).

\smallskip

Further, the concept of Heisenberg uniqueness pair has a sharp contrast to
the known results about determining sets for measures by Sitaram et al.
\cite{BS, RS}, due to the fact that the determining set $\Lambda$ for the
function $\hat\mu$ has also been considered a thin set.
In particular, if $\Gamma$ is compact, then $\hat\mu$ is a real analytic function
having exponential growth and it can vanish on a very delicate set. Hence,
in this case, finding the Heisenberg uniqueness pairs becomes little easier.
However, this question becomes immensely difficult when the measure is supported
on a non-compact curve.

\smallskip

In the article \cite{HR}, Hedenmalm and Montes-Rodr\'iguez have shown that the pair
(hyperbola, some discrete set) is a Heisenberg uniqueness pair. As a dual problem,
a weak$^\ast$ dense subspace of $L^{\infty}(\mathbb R)$  has been constructed to
solve the Klein-Gordon equation. Further, a complete characterization of the
Heisenberg uniqueness pairs corresponding to any two parallel lines has been
given by Hedenmalm and Montes-Rodr\'iguez (see \cite{HR}).

\smallskip

Afterward, a considerable amount of work has been done pertaining to the
Heisenberg uniqueness pair in the plane as well as in the Euclidean spaces.

\smallskip

Recently, Lev \cite{L} and Sj\"{o}lin \cite{S1} have independently shown that
circle and certain system of lines are HUP corresponding to the unit circle $S^1.$
Further, Vieli \cite{V1} has generalized HUP corresponding to circle in the higher
dimension and shown that a sphere whose radius does not lie in the zero set of the
Bessel functions $J_{(n+2k-2)/2};~k\in\mathbb Z_+,$ the set of non-negative integers,
is a HUP corresponding to the unit sphere $S^{n-1}.$ In a recent article \cite{Sri},
the second author has shown that a cone is a Heisenberg uniqueness pair corresponding
to sphere as long as the cone does not completely lay on the level surface of any
homogeneous harmonic polynomial on $\mathbb R^n.$ Thereafter, a sense of evidence
emerged that the exceptional sets for the HUPs corresponding to the sphere are
eventually contained in the zero sets of the spherical harmonics and the Bessel
functions, though we yet resolve it (see \cite{GR, Sri}).

\smallskip

Further, Sj\"{o}lin \cite{S2} has investigated some of the Heisenberg uniqueness
pairs corresponding to the parabola. Subsequently, Babot \cite{Ba} has given a
characterization of the Heisenberg uniqueness pairs corresponding to a certain
system of three parallel lines. Thereafter, the authors in \cite{GR} have given
some necessary and sufficient conditions for the Heisenberg uniqueness pairs
corresponding to a system of four parallel lines. However, an exact analogue
for the finitely many parallel lines is still unsolved.

\smallskip

In a major development, Jaming and Kellay \cite{JK} have given a unifying proof
for some of the Heisenberg uniqueness pairs corresponding to the hyperbola, polygon,
ellipse and graph of the functions $\varphi(t)=|t|^\alpha,$ whenever $\alpha>0.$
Further, Gr\"{o}chenig and Jaming \cite{GJ} have worked out some of the Heisenberg
uniqueness pairs corresponding to the quadratic surface.

\smallskip

Let $\Gamma$ be a finite disjoint union of smooth curves in $\mathbb R^2.$
Let $X(\Gamma)$ be the space of all finite complex-valued Borel measure
$\mu$ in $\mathbb R^2$ which is supported on $\Gamma$ and absolutely
continuous with respect to the arc length measure on $\Gamma$. For
$(\xi,\eta)\in\mathbb R^2,$ the Fourier transform of $\mu$ is defined by
\[\hat\mu{( \xi,\eta)}=\int_\Gamma e^{-i\pi(x\cdot\xi+ y\cdot\eta)}d\mu(x,y).\]
In the above context, the function $\hat\mu$ becomes a uniformly continuous bounded
function on $\mathbb R^2.$ Thus, we can analyze the pointwise vanishing nature of the
function $\hat\mu.$
\begin{definition}
Let $\Lambda$ be a set in $\mathbb R^2.$ The pair $\left(\Gamma, \Lambda\right)$
is called a Heisenberg uniqueness pair for $X(\Gamma)$ if any $\mu\in X(\Gamma)$
satisfies $\hat\mu\vert_\Lambda=0,$ implies $\mu=0.$

\end{definition}
Since the Fourier transform is invariant under translation and rotation, one can
easily deduce the following invariance properties about the Heisenberg uniqueness
pair.
\smallskip

\begin{enumerate}[(i)]
\item Let $u_o, v_o\in\mathbb R^2.$ Then the pair $(\Gamma,\Lambda)$
is a HUP if and only if the pair $(\Gamma+ u_o,\Lambda+v_o)$ is a HUP.

\smallskip

\item Let $T : \mathbb R^2\rightarrow \mathbb R^2$ be an invertible
linear transform whose adjoint is denoted by $T^\ast.$  Then  $(\Gamma,\Lambda)$
is a HUP if and only if $\left(T^{-1}\Gamma,T^\ast\Lambda\right)$  is a HUP.
\end{enumerate}
Now, we state first known results on the Heisenberg uniqueness pair due to
Hedenmalm and Montes-Rodr\'iguez \cite{HR}. Then we briefly summarize the
recent progress on this problem.

\begin{theorem}\cite{HR}\label{th9}
Let $\Gamma=L_1\cup L_2,$  where $L_j;~j=1,2$ are any two parallel straight lines
and $\Lambda$ a subset of $\mathbb R^2$ such that $\overline{\pi(\Lambda)}=\mathbb R.$
Then $(\Gamma,\Lambda)$ is a Heisenberg uniqueness pair if and only if the set
\begin{equation}\label{exp24}
\widetilde\Lambda={\pi_1^a(\Lambda)}\cup\left[{{\pi_1^b(\Lambda)}\smallsetminus{\pi_1^c(\Lambda)}}\right]
\end{equation}
is dense in $\mathbb R$.
\end{theorem}
Here we avoid mentioning the notations appeared in (\ref{exp24}) as they are bit involved,
however, we have written down the same notations as in the article \cite{HR}. Though, their
main features can be perceived in Section \ref{section3}.
\begin{theorem}\cite{HR}\label{th18}
Let $\Gamma$ be the hyperbola $x_{1}x_{2}=1$ and $\Lambda_{\alpha,\beta}$ a lattice-cross
defined by
\[\Lambda_{\alpha,\beta}=\left(\alpha\mathbb Z\times\{0\}\right)\cup\left(\{0\}\times\beta\mathbb Z\right),\]
where $\alpha, \beta$ are positive reals. Then $\left(\Gamma,\Lambda_{\alpha,\beta}\right)$ is a Heisenberg
uniqueness pair if and only if $\alpha\beta\leq1$.
\end{theorem}

For $\xi\in\Lambda,$ defining a function $e_\xi$ on $\Gamma$ by $e_{\xi}(x)=e^{i\pi x\cdot\xi}.$
As a dual problem to Theorem \ref{th18}, Hedenmalm and Montes-Rodr\'iguez \cite{HR} have proved
the following density result which in turn solve the one-dimensional Klein-Gordon equation.
\begin{theorem}\cite{HR}
The pair  $(\Gamma,\Lambda)$ is a Heisenberg uniqueness pair  if and only if
the set $\{e_{\xi}:~\xi\in\Lambda\}$ is a weak$^\ast$ dense subspace of
$L^{\infty}(\Gamma).$
\end{theorem}

Let $\Gamma$ denote a system of three parallel lines in the plane that can be
expressed as $\Gamma=\mathbb R\times \{0,1,p\},$ for some $p\in\mathbb N$ with
$p\geq2.$ The following characterization for the Heisenberg uniqueness pairs
corresponding to the above mentioned three parallel lines has been given by
D.B. Babot \cite{Ba}.

\begin{theorem}\cite{Ba}\label{th3}
Let $\Gamma= \mathbb R\times \{0,1,p\}$ and $\Lambda\subset\mathbb R^2$ a closed
set which is $2$-periodic with respect to the second variable. Then
$\left(\Gamma,\Lambda\right)$ is a HUP if and only if the set
\begin{equation}\label{exp23}
\widetilde\Lambda= {\Pi^{3}(\Lambda)}\cup\left[{{\Pi^{2}(\Lambda)}\smallsetminus{\Pi^{2^\ast}(\Lambda)}}\right]
\cup\left[{{\Pi^{1}(\Lambda)}\smallsetminus{\Pi^{1^\ast}(\Lambda)}}\right]
\end{equation}
is dense in $\mathbb R$.
\end{theorem}
For the notations appeared in (\ref{exp23}), we refer the article \cite{Ba},
as those notations are quite involved. However, they are the prototype of few
notations those will appear in Section \ref{section3}.

\smallskip

Afterward, the authors have given the following characterization for the Heisenberg
uniqueness pairs corresponding to a certain system of four parallel lines (see \cite{GR}).
Let $\Gamma$ denotes a system of four parallel straight lines in $\mathbb R^2$
which can be represented by $\Gamma=\mathbb R\times\{\alpha,\beta,\gamma,\delta\},$
where $\alpha<\beta<\gamma<\delta$ with $(\gamma-\alpha)/(\beta-\alpha)=2$ and
$(\delta-\alpha)/(\beta-\alpha)\in\mathbb N.$ By the invariance properties
of HUP, one can assume that $\Gamma=\mathbb R\times\{0,1,2,p\}$ for some $p\geq3.$

\begin{theorem}\cite{GR}\label{th4}
Let $\Gamma=\mathbb R\times\{0,1,2,p\}$ and $\Lambda\subset\mathbb R^2$ be a closed
set which is $2$-periodic with respect to the second variable. Suppose $\Pi(\Lambda)$
is dense in $\mathbb R.$ If $\left(\Gamma,\Lambda\right)$ is a Heisenberg uniqueness
pair, then the set
\[\widetilde{\Pi}(\Lambda)= {\Pi^4(\Lambda)}\bigcup\limits_{j=0}^2\left[{{\Pi^{(3-j)}(\Lambda)}\smallsetminus{\Pi^{{(3-j)}^\ast}(\Lambda)}}\right] \]
is dense in $\mathbb R.$ Conversely, if the set
\begin{equation}\label{exp2}
\widetilde{\Pi_p}(\Lambda)= {\Pi^4(\Lambda)}\bigcup\limits_{j=2}^3\left[{{\Pi^{j}(\Lambda)}\smallsetminus{\Pi^p_{{j}^\ast}(\Lambda)}}\right]\cup\left[{{\Pi^{1}(\Lambda)}\smallsetminus{\Pi^{{1}^\ast}(\Lambda)}}\right]
\end{equation}
is dense in $\mathbb R,$ then $\left(\Gamma,\Lambda\right)$ is a Heisenberg uniqueness pair.
\end{theorem}

For the notations appeared in (\ref{exp2}), we refer the article \cite{GR}.
However, the nature of their occurrence can be understood in the beginning of
Section \ref{section3} when we formulate the Heisenberg uniqueness pairs
for finitely many parallel lines.

\begin{remark}\label{rk3}
The above problem becomes immensely deeper when we consider the $p$-th line at large
that we shall elaborate in Section \ref{section3}, Remark \ref{rk4}.
\end{remark}

\section{Some individual results for Heisenberg uniqueness pairs}
In this section, we will explore some of the individual results for the
Heisenberg uniqueness pairs corresponding to the cross and certain exponential
curves in the plane as well as exponential surfaces in the Euclidean spaces.

\bigskip
Let $\Gamma=\left(\mathbb{R}\times F_n\right)\cup \left(F_n\times\mathbb{R}\right),$
where $F_n=\{0,1,\ldots,n\}.$ Consider a subset $\mathcal C\subset X(\Gamma)$ such
that for any $\mu\in\mathcal C,$ there exist $f_j\in L^1(\mathbb R);~j\in F_n$
satisfying
\begin{equation}\label{exp123}
d\mu(x,y)=\sum\limits_{k=0}^n\left(f_k(x)dxd\delta_k(y)-f_k(y)d\delta_k(x)dy\right),
\end{equation}
where $\delta_k$ denotes the point mass measure at $k.$ Then the pair $(\Gamma,\Lambda)$
is said to be a $\mathcal C$-HUP if any measure $\mu\in \mathcal C$
satisfies $\hat\mu\vert_\Lambda=0,$ implies $\mu$ is identically zero.

\smallskip

For $n=0,$ let $\alpha_j\in\mathbb R;~j=1,2$ be independent over $\mathbb{Q},$ the set of
rational numbers. Consider the lines $L_{\alpha_j}=\{(\xi,\eta)\in\mathbb{R}^2:\eta-\xi=\alpha_j\};~j=1,2.$
\begin{theorem}
Let $\Gamma=\left(\mathbb{R}\times\{0\}\right)\cup\left(\{0\}\times\mathbb{R}\right)$ and $\Lambda=L_{\alpha_1}
\cup L_{\alpha_2},$ where $\alpha_j;~ j=1,2$ are independent over $\mathbb{Q}.$ Then
$(\Gamma,\Lambda)$ is a $\mathcal C$-HUP.
\end{theorem}

\begin{proof}
For $\mu\in\mathcal C$ there exist $f\in L^{1}(\mathbb R)$ such that
\[d\mu(x,y)=f(x)dxd\delta_o(y)-f(y)d\delta_o(x)dy.\]
By taking the Fourier transform of both the sides, we get
\begin{eqnarray*}
\hat\mu(\xi,\eta) &=& \int_\Gamma e^{-i\pi(x\xi+y\eta)}d\mu(x,y)\\
&=& \int_\mathbb{R} e^{-i\pi x\xi}f(x)dx-
\int_\mathbb{R} e^{-i\pi y\eta}f(y)dy\\
&=& \hat f(\xi)-\hat f(\eta).
\end{eqnarray*}
Now, $\hat\mu|_{L_{\alpha_1}}=0$ implies that $\hat f(t)=\hat f(t+\alpha_1)$
for all $t\in\mathbb{R}.$ Using the invariance property $(ii),$ we can assume
that $\alpha_j>0;~ j=1,2.$ Hence, $\hat f$ is a periodic function with some
period $p_o>0.$ By the simple recursions we can derive that $\hat f(t)=\hat f(t+m\alpha_1)$
for all $m\in\mathbb Z$ and $t\in\mathbb R.$ Similarly, by the condition $\hat\mu|_{L_{\alpha_2}}=0,$
it follows that $\hat f(t)=\hat f(t+n\alpha_2)$ for all $n\in\mathbb Z$ and $t\in\mathbb R.$
Thus, we infer that
\[\hat f(t)=\hat f(t+m\alpha_1+n\alpha_2)\]
whenever, $m,n\in\mathbb Z$ and $t\in\mathbb R.$ By Kr\"{o}necker approximation
theorem (see \cite{ATM}), we can always find $m_o,n_o\in\mathbb{Z}$ such that $|m_o\alpha_1+n_o\alpha_2|<p_o,$ which
contradict the fact that $p_o$ is the period of $f.$  This, in turn, implies that
$\hat f$ must be a constant function. Thus, by Riemann-Lebesgue lemma $f=0.$
\end{proof}

\begin{remark}\label{rk1}
$(i).$ Let $\Gamma_o=\left(\mathbb R\times\{0\}\right)\cup\left(\{0\}\times\mathbb R\right)$ and
$\Lambda=\{(\xi,\eta)\in\mathbb{R}^2:\eta=\alpha\xi,~|\alpha|<1\}.$
Then by Riemann-Lebesgue lemma, it can be easily deduced that $(\Gamma_o,\Lambda)$
is a $\mathcal C$-HUP. On the other hand, if $\Lambda\subseteq\{(\xi,\eta)\in\mathbb R^2:\eta=\xi\},$
then $(\Gamma_o,\Lambda)$ is not a $\mathcal C$-HUP. Next, for $\Lambda=\{(\xi,\eta)\in\mathbb{R}^2:\eta=\xi^2\},$
one can easily verify that $(\Gamma_o,\Lambda)$ is a $\mathcal C$-HUP.

\smallskip

$(ii).$ Consider $\Gamma=\left(\mathbb{R}\times F_n\right)\cup\left(F_n\times\mathbb R\right)$
and $\Lambda=\{(\xi,\eta) \in\mathbb{R}^2:\eta=\xi\}.$ Then $(\Gamma,\Lambda)$ is not a
$\mathcal C$-HUP. For this, we can choose non-zero functions $f_j\in L^1(\mathbb R);~j\in F_n$
such that
\begin{equation}
\hat\mu(\xi,\eta)=\sum\limits_{k=0}^n \left(e^{-ik\pi\eta}\hat f_k(\xi)-e^{-ik\pi\xi}\hat f_k(\eta)\right)
\end{equation}
and hence $\hat\mu(\xi,\eta)=0,$ whenever $\xi=\eta.$ For $n=1,$ if $\Lambda$
is the $\xi$-axis, then $(\hat f_0+\hat f_1)(\xi)=\hat f_0(0)+e^{-i\pi\xi}\hat f_1(0)$
for all $\xi\in\mathbb R.$ By Riemann-Lebesgue lemma, it is enough to consider
$f_1=-f_0$ and $\hat f_0(0)=0.$ Thus, $(\Gamma,\Lambda)$ is not a $\mathcal C$-HUP.
\end{remark}

$(iii).$  In view of injectivity of the Fourier transform on $L^1(\mathbb R),$
for a pair $(\Gamma_o, \Lambda)$ to be a $\mathcal C$-HUP it is necessary that at least
one of the orthogonal projections $\pi_j(\Lambda);~j=1,2$ on the axes must be
dense in  $\mathbb R.$ It would be an interesting question to get a sufficient
condition for the $\mathcal C$-Heisenberg uniqueness pairs corresponding to $\Gamma_o.$

\smallskip

Next, we will work out some of the Heisenberg uniqueness pairs corresponding to
certain exponential curve and surfaces in the Euclidean spaces.
\smallskip

Let $\mu$ be a finite Borel measure having support on $\Gamma=\{(t, e^{t^2}): t\in\mathbb R\}$
which is absolutely continuous with respect to the arc length on $\Gamma.$ Then there exists
$f\in L^{1}(\mathbb R)$ such that $d\mu=g(t)dt,$ where $g(t)=f(t)\sqrt{1+4t^2e^{2t^{2}}}.$
Hence by finiteness of $\mu,$ it follows that $g\in L^{1}(\mathbb R)$ and
\begin{equation}\label{exp4}
\hat\mu{(x, y)}=\int_{\mathbb R}e^{- i\pi\left(xt + ye^{t^2}\right)}g(t)dt.
\end{equation}
Then $\hat\mu$ satisfies the PDE
\begin{equation}
\left(1+\mathcal T_x\right)\hat\mu=\alpha\partial_y\hat\mu
\end{equation}
in the distributional sense, where $\mathcal T_x=\sum\limits_{n=0}^\infty\left(\sum\limits_{k=1}^4\dfrac{\alpha^{4n+k}}{(4n+k)!}\partial_x^{4n+k}\right)$
and $\alpha=\frac{i}{\pi}.$

\begin{theorem}\label{th08}
Let $\Gamma=\left\{\left(t, \alpha(t)\right):~ t\in\mathbb R\right\},$ where
$\alpha :\mathbb R\rightarrow\mathbb R_{+}$ be defined by $\alpha(t)=e^{t^2}.$
\smallskip

$(i).$ Let $\Lambda$ be a straight line. Then $\left(\Gamma,\Lambda\right)$
is a Heisenberg uniqueness pair if and only if $\Lambda$ is parallel to the $x$-axis.
\smallskip

$(ii).$ Let $L_j;~j=1,2$ be two parallel lines which are not parallel to either of the axes.
Then $\left(\Gamma,L_1\cup L_2\right)$ is a Heisenberg uniqueness pair.
\end{theorem}

In order to prove Theorem \ref{th08}, we need the following results. First,
we state a result which can be found in Havin and J\"{o}ricke \cite{HJ} at p. 36.
\begin{lemma}\cite{HJ}\label{lemma3}
Let $\varphi\in L^1[0,\infty).$ If $\int\limits_{\mathbb R}\log|\hat\varphi|\dfrac{dx}{1+x^2}=-\infty,$
then $\varphi=0.$
\end{lemma}
Next, we state the following form of Radon-Nikodym derivative theorem (see \cite{F}, p.91).
\begin{proposition}\label{prop3}
Let $\nu$ be a $\sigma$-finite signed measure which is absolutely continuous with respect to
a $\sigma$-finite measure $\mu$ on the measure space $(X,\mathcal M).$ If $g\in L^1(\nu),$
then $g\frac{d\nu}{d\mu}\in L^1(\mu)$ and $\int gd\nu=\int g\frac{d\nu}{d\mu}d\mu.$
\end{proposition}
As a consequence of Lemma \ref{lemma3} and Proposition \ref{prop3}, we prove the
following result. Let $h_c :\mathbb{R}\rightarrow\mathbb{R}$ defined by
$h_c(t)=\left(t+ \frac{1}{2c}\right)^2+e^{t^2}-t^2,$ where $c$ is a non-zero constant.
Let $|E|$ denotes the Lebesgue measure of the set $E\subset\mathbb R.$

\begin{lemma}\label{lemma105}
Let $g\in L^{1}(\mathbb R)$. Suppose $E\subset\mathbb R$ and $|E|>0.$ Then
\begin{equation}\label{exp105}
\int_{\mathbb R}e^{-i\pi cxh_c(t)}g(t)dt=0
\end{equation}
for all $x\in E$ if and only if  $\psi_c(u)=g\circ h_c^{-1}(u^2)\frac{2u}{h_c'\circ h_c^{-1}(u^2)}$
is an odd function.
\end{lemma}

\begin{proof}
By  Proposition \ref{prop3}, we can write the left-hand side of  (\ref{exp105}) as
\begin{eqnarray}
I&=&\int_{\mathbb R}e^{-i\pi cxh_c(t)}g(t)dt\\
&=&\int_{\mathbb R}e^{-i\pi cxu^2}g\circ h_c^{-1}(u^2)\frac{2udu}{h_c'\circ h_c^{-1}(u^2)}  \nonumber \\
&=&\int_{\mathbb R}e^{-i\pi cxu^2}\psi_c(u)du.\nonumber
\end{eqnarray}
Hence by Proposition \ref{prop3} it follows that $\psi_c\in L^{1}(\mathbb{R})$ and we get
\begin{eqnarray*}
I&=&\int^{\infty}_{0}e^{-i\pi cxu^2}\psi_c(u)du+\int_{0}^{\infty}e^{-i\pi cxu^2}\psi_c(-u)du\\
&=&\int_{0}^{\infty}e^{-i\pi cxu^2}F_c(u)du,
\end{eqnarray*}
where $F_c(u)=\psi_c(u)+\psi_c(-u)$ for all $u>0.$ Clearly $F_c\in L^{1}(0,\infty)$ and
by the change of variables $u^2=v,$ we have
\begin{equation}\label{exp16}
I=\int_0^{\infty}e^{-i\pi cxv}F_c(\sqrt{v})\frac{dv}{2\sqrt{v}}.
\end{equation}
Let $\varphi(v)={F_c(\sqrt{v})}/{2\sqrt{v}}~\chi_{(0,\infty)}(v).$
Then $\varphi\in L^{1}(\mathbb R)$ and from (\ref{exp16}) we obtain
$I=\hat{\varphi}(cx)=0$ for all $x\in E.$ That is, $\hat{\varphi}$
vanishes on the set $cE$ of positive measure. Thus, by Lemma \ref{lemma3}
we conclude that $\varphi=0$ a.e. Hence, it follows that $F_c=0$ a.e. on
$[0,\infty).$
\smallskip

Conversely, if $\psi_c$ is an odd function, then (\ref{exp105}) trivially holds.

\end{proof}
\smallskip
As a corollary to Lemma \ref{lemma105}, we can easily derive the following result.
\begin{corollary}\label{cor1}
Let $g\in L^{1}(\mathbb R)$ and $\alpha :\mathbb R\rightarrow\mathbb R_{+}$ be
defined by $\alpha(t)=e^{t^2}$. Suppose $E\subset\mathbb R$ and $|E|>0.$ Then
\begin{equation}\label{exp15}
\int_{\mathbb R}e^{-i\pi x\alpha(t)}g(t)dt=0
\end{equation}
for all $x\in E$ if and only if  $g$ is an odd function.
\end{corollary}

\smallskip
\begin{proposition}\label{prop125}
Suppose $E\subset\mathbb R$ and $|E|>0.$ Assume $c,d\in\mathbb{R}$
with $c,d\neq0.$ Then \\
(i) $\widehat\mu(x,cx)=0$ for all $x\in E$ if and only if $\psi_c$ is an odd function.\\
(ii) $\widehat\mu(x,cx+d)=0$ for all $x\in E$ if and only if
$\phi_c(u)=\chi\circ h_c^{-1}(u^2)\dfrac{2u}{h'_c\circ h_c^{-1}(u^2)}$
is an odd function of $u,$ where $\chi(t)=e^{- i\pi de^{t^2}}g(t).$
\end{proposition}
\begin{proof}
$(i).$ From (\ref{exp4}) we can express
\[\hat\mu{(x, cx)}=\int_{\mathbb R}e^{- i\pi x\left(t + ce^{t^2}\right)}g(t)dt
=e^{i\pi x/4c}\int_{\mathbb R}e^{- i\pi cxh_c(t)}g(t)dt.\]
By Lemma \ref{lemma105}, $\psi_c$ is odd if and only if $\widehat\mu(x,cx)=0$
for all $x\in E.$
\smallskip

$(ii).$ By a simple computation, we get
\[\hat\mu{(x, cx+d)}=\int_{\mathbb R}e^{- i\pi x\left(t + ce^{t^2}\right)}\chi(t)dt
= e^{i\pi x/4c}\int_{\mathbb R}e^{- i\pi cxh_c(t)}\chi(t)dt.\]
As similar to the above case, $\phi_c$ is odd if and only if $\widehat\mu(x,cx+d)=0$
for all $x\in E.$
\end{proof}
\smallskip

\begin{proof}[ Proof of Theorem \ref{th08}]
$(i).$  In view of the invariance property $(\text{i}),$ we can assume that $\Lambda$ is
the $x$-axis. Recall that $\hat\mu$ satisfies
\[\hat\mu{(x, y)}= \int_{\mathbb R}e^{- i\pi\left(xt + ye^{t^2}\right)}g(t)dt.\]
Hence $\hat\mu\vert_\Lambda=0$ implies that $\hat g(x)=0$ for all $x\in\mathbb R.$
Thus, we conclude that $\mu=0.$
\smallskip

Conversely, suppose $\Lambda$ is not parallel to the $x$-axis. If $\Lambda$ is parallel to
the $y$-axis, then by Corollary \ref{cor1}, it follows that $\left(\Gamma,\Lambda\right)$
is not a HUP. Hence we can assume that $\Lambda$ of the form
$y=cx,$ where $c\neq0.$ Choose a non-zero odd function $\varphi\in L^{1}(\mathbb R)$ and
let $g(t)=\frac{\varphi(\sqrt{h_c(t)})h'_c(t)}{2\sqrt{h_c(t)}}.$ Then by Proposition \ref{prop125},
it follows that $\left(\Gamma,\Lambda\right)$ is not a Heisenberg uniqueness pair.

\smallskip

$(ii).$ Let $L_1=\{(x, cx):~x\in\mathbb R\}$ and $L_2=\{(x, cx+d):~x\in\mathbb R\}$, where $c,d\neq0.$
Since $\hat\mu|_{L_j}=0;~j=1,2,$ by Proposition \ref{prop125} it follows that $\psi_c$ and $\phi_c$ are
odd functions. Let $u_+$ and $u_-$ be the square roots of $h_c.$ Since $\phi_c$ is an odd function, it implies that
\[\left[e^{i\pi d\left\{e^{\left(h_c^{-1}(u_-^2)\right)^2}-e^{\left(h_c^{-1}(u_+^2)\right)^2}\right\}}-1\right]\psi_c(u)=0.\]
That is, $\psi_c=0$ a.e. and hence $g\circ h_c^{-1}(u^2)=0$ almost all $u.$
Thus, the pair $(\Gamma,L_1\cup L_2)$ is a Heisenberg uniqueness pair.

\end{proof}

\begin{remark}\label{rk2}
Let $\alpha :\mathbb R\rightarrow\mathbb R_{+}$ be an even smooth function having
finitely many local extrema and consider $\Gamma=\left\{(t, \alpha(t)): t\in\mathbb R\right\}.$
Then the conclusions of Theorem \ref{th08} would also hold.
\end{remark}

\smallskip

\begin{theorem}\label{th10}
Let $\Gamma$ be the surface $x_{n+1}=e^{x_1^2}+\cdots+e^{x_n^2}$ in $\mathbb R^{n+1}$ and $\Lambda$ an
affine hyperplane in $\mathbb R^{n+1}$ of dimension $n.$ Then $\left(\Gamma,\Lambda\right)$
is a Heisenberg uniqueness pair if and only if $\Lambda$ is parallel to the hyperplane $x_{n+1}=0.$
\end{theorem}

For $u=(u_1,\ldots,u_n),$ denoting $\varphi(u)=e^{u_1^2}+\cdots+e^{u_n^2}.$ Let $\mu$ be a finite Borel
measure which is supported on $\Gamma=\{(u,\varphi(u)):~u\in\mathbb R^n\}$ and absolutely continuous
with respect to the surface measure on $\Gamma.$  Then by Radon-Nikodym theorem, there exists a
measurable function $f$ on $\mathbb R^n$ such that $d\mu=g(u)du,$ where
$g(u)=f(u)\sqrt{1+|\text{grad}~\varphi(u)|^2}.$ Then by the finiteness of $\mu,$ it follows that
$g\in L^1(\mathbb R^n).$ Denote $u'=(u_2,\ldots,u_n)$ and $x'=(x_2,\ldots,x_n).$ Then the Fourier
transform of $\mu$ can be expressed as
\begin{equation}\label{exp70}
\hat\mu(x)=\int_{\mathbb R^n}e^{-\pi i\left(x'.u'+x_{n+1}\varphi(u)\right)}g(u)du
\end{equation}
for $x\in\mathbb R^{n+1}.$
\smallskip

\begin{proof}[ Proof of Theorem \ref{th10}.]
Since $\Lambda$ is an affine hyperplane in $\mathbb R^{n+1}$ of dimension $n,$ by
 the invariance properties of HUP, we can assume that $\Lambda$ is a linear
subspace of $\mathbb R^{n+1}$ which can be considered as either $x_{n+1}=cx_1,$
where $c\in\mathbb R$ or $x_1=0.$
\smallskip

If $\Lambda=\left\{(x_1,\ldots,x_{n+1})\in\mathbb R^{n+1}:x_{n+1}=0\right\},$
then by the hypothesis, $\hat\mu\vert_\Lambda=0$ implies $\hat g=0$ on $\mathbb R^n.$
Thus, it follows that $\left(\Gamma,\Lambda\right)$ is a HUP.

\smallskip

Conversely, suppose $\Lambda$ is not parallel to the hyperplane $x_{n+1}=0.$ Consider
a non-zero compactly supported odd function $\psi\in L^1(\mathbb R)$ together with a
non-zero compactly supported  function $h\in L^1(\mathbb R^{n-1}).$ Then we have the
following two cases.
\smallskip

$(i).$ If $\Lambda=\left\{(x_1,\ldots,x_{n+1})\in\mathbb R^{n+1}:x_1=0\right\}$
and $g(u)=\psi(u_1)h(u').$ Then for $x\in\Lambda,$ we have
\begin{eqnarray*}
\hat\mu(x) &=&\int_{\mathbb R^n}e^{-\pi i\left(x'.u'+x_{n+1}\varphi(u)\right)}g(u)du\\
&=&\int_{\mathbb R^n}e^{-\pi i\left(x'.u'+x_{n+1}\varphi(u)\right)}\psi(u_1)h(u')du\\
&=&\int_{\mathbb R^{n-1}}e^{-\pi i\left\{x'.u'+x_{n+1}\left(\varphi(u)-e^{u_1^2}\right)\right\}}
\left(\int_\mathbb Re^{-\pi ix_{n+1}e^{u_1^2}}\psi(u_1)du_1\right)h(u')du'\\
&=&0.
\end{eqnarray*}
Thus, $\left(\Gamma,\Lambda\right)$ is not a Heisenberg uniqueness pair.

\smallskip

$(ii).$ Suppose $\Lambda=\left\{(x_1,\ldots,x_{n+1})\in\mathbb R^{n+1}:x_{n+1}=cx_1,c\neq0\right\}.$
Consider a function $\tau(t)=\frac{\psi\left(\sqrt{h_c(t)}\right)h'_c(t)}{2\sqrt{h_c(t)}}$
and write $g(u)=\tau(u_1)h(u').$ If we denote $x''=(x_1,\ldots,x_n).$ Then for $x\in\Lambda,$ we have
\begin{align*}
\hat\mu(x) &= \int_{\mathbb R^n}e^{-\pi i\left(x''.u+cx_1\varphi(u)\right)}g(u)du\\
&= \int_{\mathbb R^n}e^{-\pi i\left(x''.u+cx_1\varphi(u)\right)}\tau(u_1)h(u')du\\
&= \int_{\mathbb R^{n-1}}e^{-\pi i\left\{x'.u'+cx_1\left(\varphi(u)-e^{u_1^2}\right)\right\}}
\left(\int_\mathbb Re^{-\pi i\left(x_1u_1+cx_1e^{u_1^2}\right)}\tau(u_1)du_1\right)h(u')du'.
\end{align*}
By Lemma \ref{lemma105}, it follows that
\begin{eqnarray*}
\int_\mathbb Re^{-\pi ix_1(u_1+ce^{u_1^2})}\tau(u_1)du_1&=&\int_\mathbb Re^{-\pi icx_1h_c(u_1)}\tau(u_1)du_1\\
&=&\int_\mathbb Re^{-\pi icx_1t^2}\psi(t)dt\\
&=&0.
\end{eqnarray*}
Thus, we conclude that $\left(\Gamma,\Lambda\right)$ is not a Heisenberg uniqueness pair.
\end{proof}

\section{Heisenberg uniqueness pairs corresponding to the finite number of parallel lines}\label{section3}
A characterization of the Heisenberg uniqueness pairs corresponding to two parallel
lines was being done by Hedenmalm and Montes-Rodr\'iguez  (see \cite{HR}). Afterward,
Babot \cite{Ba} worked out an analogous result for a certain system of three parallel
lines. Further, authors had extended the above result to a system of four parallel
lines (see \cite{GR}). In this section, we shall prove a characterization of the
Heisenberg uniqueness pairs corresponding to a certain system of finitely many
parallel lines. In the latter case, we observe the phenomenon of interlacing
of the finitely many totally disconnected sets.
\smallskip

Denote $F_s=\{0,1,\ldots, s\}.$ Let $\Gamma_o$ be a system of $n+2$ parallel lines that
can be expressed as $\Gamma_o=\mathbb R\times\{\alpha_o,\ldots,\alpha_{n+1}\},$ where
$\alpha_o<\cdots<\alpha_{n+1}$ and  $p_j=(\alpha_j-\alpha_o)/(\alpha_1-\alpha_o)$ with
\[ p_j=\begin{cases}
      j & \text{if}~2\leq j\leq n, \\
      p\in\mathbb N\smallsetminus F_n &\text{if}~ j=n+1.
   \end{cases}
\]
If $(\Gamma_o, \Lambda_o)$ is a HUP, then by the invariance property $(i),$ $(\Gamma_o, \Lambda_o)$
can be thought as $\left(\Gamma_o-(0,\alpha_o), \Lambda_o\right).$ Since scaling in $\mathbb R^2$
can be viewed as a diagonal matrix, by the invariance property $(ii),$ $\left(\Gamma_o-(0,\alpha_o), \Lambda_o\right)$
can be considered as $\left(T^{-1}(\Gamma_o-(0,\alpha_o)), T^\ast\Lambda_o\right),$ where
$T=\text{diag}\{(\alpha_1-\alpha_o),~(\alpha_1-\alpha_o)\}.$ Let $\Gamma=T^{-1}(\Gamma_o-(0,\alpha_o))$
and $\Lambda=T^\ast\Lambda_o.$ Then $\Gamma=\mathbb R\times\left(F_n\cup\{p\}\right),$ where $p\in\mathbb N$
with $p\geq n+1.$ In view of the above facts, $(\Gamma_o, \Lambda_o)$ is a
HUP if and only if $(\Gamma, \Lambda)$ is a HUP.

\smallskip

To state our main result of this section,  we need to set up some necessary
notations and few auxiliary results.

\smallskip

Let $\mu\in X(\Gamma).$ Then there exist $f_k\in L^1(\mathbb R);~k\in F_{n+1}$
such that
\begin{equation}\label{exp1}
d\mu(x,y)=\sum_{j=0}^nf_j(x)dxd\delta_j(y)+f_{n+1}(x)dxd\delta_p(y),
\end{equation}
where $\delta_{t}$ denotes the point mass measure at $t.$ For $(\xi,\eta)\in\mathbb R^2,$
the Fourier transform of $\mu$ can be written as
\begin{equation}\label{exp11}
\hat\mu(\xi,\eta)=\sum_{j=0}^ne^{j\pi i\eta}\hat{f_j}(\xi) + e^{p\pi i\eta}\hat{f}_{n+1}(\xi).
\end{equation}
Since $\mu$ is supported on $\Gamma=\left\{(x,y)\in\mathbb R^2: P(y)\equiv\prod_{k=0}^n(y-k)(y-p)=0\right\},$
 the function $\hat\mu$ will satisfy the PDE
\begin{equation}\label{exp81}
P\left(\frac{1}{\pi i}\frac{\partial}{\partial\eta}\right)\hat\mu=0
\end{equation}
in the distributional sense with the initial condition $\hat\mu\vert_{\Lambda}=0.$
On the other hand, if (\ref{exp81}) holds, then $\text{supp}~\mu\subset\Gamma.$

\begin{remark}\label{rk4}
Notice that for each fixed $(\xi,\eta)\in\Lambda\subset\mathbb R^2,$ the right-hand side of
$(\ref{exp11})$ is a trigonometric polynomial of degree $p$ that could have preferably some
missing terms. Therefore, it is an interesting question to find out the smallest determining
set $\Lambda$ for the above trigonometric polynomial. We observe that the size of $\Lambda$
depends on the choice of number of lines as well as irregular separation among themselves.
That is, larger number of lines or value of $p$ would force smaller size of $\Lambda.$
Eventually, the problem would become more difficult for large value of $p.$
\end{remark}
\smallskip

Observe that $\hat \mu$ is a $2$-periodic function in the second variable.
Thus, for any set $\Lambda\subset\mathbb R^2,$ it is enough to consider the set
\[\pounds(\Lambda)= \left\{(\xi,\eta) : (\xi,\eta+ 2k)\in\Lambda, \text{ for some } k\in\mathbb Z\right\}\]
for the purpose of HUP. Also, it is easy to verify that $\left(\Gamma,\Lambda\right)$ is a
HUP if and only if $(\Gamma,\overline{\pounds(\Lambda)})$ is a HUP, where
$\overline{\pounds(\Lambda)}$ denotes the closure of $\pounds(\Lambda)$ in
$\mathbb R^2.$ Thus, it is enough to work with the closed\texttt{}
set $\Lambda\subset\mathbb R^2$ which is $2$-periodic in the second variable.

\begin{remark}\label{rk5}
Further, it is evident from the Riemann-Lebesgue lemma that the exponential
functions, which appeared in (\ref{exp11}), can not be expressed as
the Fourier transform of functions in $L^1(\mathbb R).$ However, by
Wiener's lemma, they can be locally agreed with the Fourier transform
of functions in $L^1(\mathbb R).$
\end{remark}

Hence, in view of Remark \ref{rk5} and the condition
$\hat\mu\vert_\Lambda=0,$ we can classify these exponential functions in
the form of the following annihilating spaces which will create dispensable
sets in the determining set $\Lambda.$

Given $E\subset\mathbb R$ and a point $\xi\in E,$ let $I_\xi$ be an interval
containing $\xi.$ We define

\smallskip

\noindent $\textbf{(A1).}$ $ L^{E,\xi}_{loc}=\{\psi :E\rightarrow\mathbb C$
such that there is an interval $I_\xi$ and a function $\varphi\in L^1(\mathbb R)$
satisfying $\psi= \hat{\varphi}$ on $I_{\xi}\cap E\}.$
\smallskip

Next, for $2\leq m\leq n,$ we define
\smallskip

\noindent $\textbf{(A2).}$
$P^{1, m}[ L^{E,\xi}_{loc}]=\{\psi :E\rightarrow\mathbb C$ such that there is an interval
$I_{\xi}$ and functions $\varphi_j\in L^1(\mathbb R);~j\in F_{m-1}$ satisfying
$\psi^m+\hat{\varphi}_{m-1}\psi^{m-1}+\cdots+\hat{\varphi}_1\psi+\hat{\varphi}_0=0 \text { on }I_{\xi}\cap E\}.$
\smallskip

Finally, we define the $p$-th space by
\smallskip

\noindent $\textbf{(A3).}$  $P^{1, p}[ L^{E,\xi}_{loc}]=\{\psi :E\rightarrow\mathbb C$
such that there is an interval $I_{\xi}$ and functions $\varphi_j\in L^1(\mathbb R);~j\in F_n$
satisfying $\psi^p+\hat{\varphi}_n\psi^n+\cdots+\hat{\varphi}_1\psi+\hat{\varphi}_o=0$
on $I_{\xi}\cap E\}.$
\smallskip

We frequently need the following Wiener's lemma that will be a key ingredient
for most of the proofs in the sequel.

\begin{lemma}\label{lemma4}\cite{K}
Let $\psi\in L^{E,\xi}_{loc}$ and $\psi(\xi)\neq0.$ Then $1/{\psi}\in  L^{E,\xi}_{loc}.$
\end{lemma}
For more details see \cite{K}, p.57.
\smallskip

Next, we establish a relation among the sets defined by $\textbf{(A1)},$  $\textbf{(A2)}$
and $\textbf{(A3)}$ that will help in figure out the dispensable sets.

\begin{lemma}\label{lemma5} The following inclusions hold.
\begin{equation}\label{exp3}
L^{E,\xi}_{loc}\subset P^{1,2}[ L^{E,\xi}_{loc}]\subset\cdots\subset P^{1,n}[ L^{E,\xi}_{loc}]\subset P^{1, p}[ L^{E,\xi}_{loc}].
\end{equation}
\end{lemma}
For proving Lemma \ref{lemma5}, we need the following result which is appeared in \cite{GR}.
\begin{lemma}\cite{GR}\label{lemma20} For $p\geq3,$ the following inclusions hold.
\begin{equation}\label{exp101}
L^{E,\xi}_{loc}\subset P^{1,2}[ L^{E,\xi}_{loc}]\subset P^{1, p}[ L^{E,\xi}_{loc}].
\end{equation}
\end{lemma}
\smallskip

\begin{proof}[ Proof of Lemma \ref{lemma5}]
$\bf{(a).}$ By Lemma \ref{lemma20}, it follows that $P^{1,2}[L^{E,\xi}_{loc}]\subset P^{1,3}[L^{E,\xi}_{loc}].$
Next, we assume that $P^{1,m-1}[L^{E,\xi}_{loc}]\subset P^{1, m}[L^{E,\xi}_{loc}],$ whenever $3\leq m<n+1.$
If $\psi\in P^{1, m}[L^{E,\xi}_{loc}],$ then there exists $I_\xi$ containing $\xi$ and
$h_j\in L^1(\mathbb R);~j\in F_{m-1}$ such that
\begin{equation}\label{exp102}
\psi^m+\hat{h}_{m-1}\psi^{m-1}+\cdots+\hat{h}_1\psi+\hat{h}_0=0
\end{equation}
on $I_\xi\cap E.$
Now, consider a function $\chi\in L^1(\mathbb R)$ such that $I_\xi\subset\text{supp }\hat\chi.$
After multiplying (\ref{exp102}) by $\psi$ and $\hat\chi$ separately
and adding the resultant equations, we can write
\[\psi^{m+1}+\left(\hat\chi+ \hat h_{m-1}\right)\psi^{m}+\cdots+\left(\hat\chi\hat h_1+
\hat h_o\right)\psi+\hat\chi\hat h_o=0.\]
Hence, for the appropriate choice of $\varphi_j;~j\in F_m,$ we have
\[\psi^{m+1}+\hat{\varphi}_m\psi^m+\cdots+\hat{\varphi}_1\psi+\hat{\varphi}_0=0\]
on $I_\xi\cap E.$ Thus $\psi\in P^{1, m+1}[ L^{E,\xi}_{loc}].$
\smallskip

\noindent $\bf{(b).}$ It only remains to show that $P^{1, n}[ L^{E,\xi}_{loc}]\subset P^{1,p}[ L^{E,\xi}_{loc}].$
If $\psi\in P^{1, n}[ L^{E,\xi}_{loc}],$ then there exists $I_\xi$ containing $\xi$
and $\varphi_j\in L^1(\mathbb R);~j\in F_n$ such that
\begin{equation}\label{exp17}
\psi^{n+1}+\hat{\varphi}_{n}\psi^{n}+\cdots+\hat{\varphi}_1\psi+\hat{\varphi}_0=0
\end{equation}
on $I_\xi\cap E.$ After multiplying (\ref{exp17}) by $\psi$ and $\hat{\varphi}_n$ separately,
 we can write
\[\psi^{n+2}+\left(\hat{\varphi}_{n-1}-\hat\varphi_{n}\hat\varphi_{n}\right)\psi^{n}+\cdots+
\left(\hat{\varphi}_0-\hat\varphi_1\hat\varphi_{n}\right)\psi+\left(-\hat\varphi_0\hat\varphi_{}
\right)=0.\]
Hence, for the appropriate choice of $\varphi_j\in L^1(\mathbb R);~j\in F_n$ and
induction hypothesis, we conclude that
\[\psi^{p}+\hat{\varphi}_n\psi^n+\cdots+\hat{\varphi}_1\psi+\hat{\varphi}_0=0\]
on $I_\xi\cap E.$ Thus $\psi\in P^{1, p}[ L^{E,\xi}_{loc}].$
\end{proof}

Let $\Pi(\Lambda)$ be the projection of $\Lambda$ on the $\xi$ - axis.
For partitioning the projection $\Pi(\Lambda),$ we introduce the notion of
symmetric polynomials. Since the level surface of a complete homogeneous
symmetric polynomial consists of only finitely many those points which can
be represented by an upper triangular matrix. More details will be given in
\textbf{Appendix A}.
\smallskip

For each $k\in\mathbb Z_+,$ the complete homogeneous
symmetric polynomial $h_k$ of degree $k$ is the sum of all monomials of
degree $k.$ That is,
\[h_k\left(x_1,\ldots, x_s\right)=\sum\limits_{\alpha_1+\cdots+\alpha_s=k}x_1^{\alpha_1}\cdots x_s^{\alpha_s}.\]

\smallskip

Given $\xi\in\Pi(\Lambda),$ we denote the corresponding images on the $\eta$ - axis by
\[\varSigma_{\xi}= \{\eta\in [0,2) : (\xi,\eta)\in\Lambda\}.\]
Since it is expected that the set $\varSigma_{\xi }$ may consist more than one
images which depend upon the order of winding in $\Lambda$ around the line
$\{\xi\}\times\mathbb R,$ the set $\Pi (\Lambda )$ can be decomposed into the
following disjoint sets.

\smallskip

\noindent ${\bf{(P_1).}}$ $\Pi^1(\Lambda)=\{\xi\in\Pi(\Lambda): \text{ there is unique } \eta_0\in \varSigma_{\xi}\}.$
\smallskip

Next, for $2\leq m\leq n$ we define the following $n-1$ partitioning sets.
\smallskip

\noindent ${\bf{(P_m).}}$ $\Pi^m(\Lambda)=\{\xi\in\Pi(\Lambda):\text{there are exactly $m$ distinct}~\eta_j\in\varSigma_{\xi};~j\in F_{m-1}\}.$
\smallskip

For $n+2$ distinct images $\eta_j\in[0,2),$ denote $a_j=e^{\pi i\eta_j};~j\in F_{n+1}.$
Then, in view of \textbf{Appendix A}, the remaining partitioning sets can be described
as follows.
\smallskip

\noindent ${\bf{(P_{n+1}).}}$ $\Pi^{n+1}(\Lambda)=\{\xi\in\Pi(\Lambda):$ there are at least
$n+1$ distinct $\eta_j\in\varSigma_{\xi}$ for $j\in F_n$ and if there is another
$\eta_{n+1}\in \varSigma_{\xi},$ then $h_{p-n}(a_0,\ldots,a_{n-1},a_n)=h_{p-n}(a_0,\ldots,a_{n-1},a_{n+1})\}.$
\smallskip

\noindent ${\bf{(P_{n+2}).}}$ $\Pi^{n+2}(\Lambda)=\{\xi\in\Pi(\Lambda):$ there are at least
$n+2$ distinct $\eta_j\in\varSigma_{\xi}$ for $j\in F_{n+1}$ satisfies
\begin{equation}\label{exp12}
h_{p-n}(a_0,\ldots,a_{n-1},a_n)\neq h_{p-n}(a_0,\ldots,a_{n-1},a_{n+1})\}.
\end{equation}

In this way, we have a decomposition $\Pi(\Lambda)=\bigcup\limits_{j=1}^{n+2}\Pi^j(\Lambda).$

\smallskip

Next, we derive some of the auxiliary results related to $\Pi^m(\Lambda);~2\leq m\leq n+1.$
For $m$ distinct images $\eta_j\in [0,2),$ we write $a_j=e^{\pi i\eta_j};~j\in F_{m-1}.$
Consider the system of equations
\begin{equation}\label{exp05}
A^m_\xi X=B^m_\xi,
\end{equation}
where $b_s=(a_o^s,\ldots,a_{m-1}^s);~0\leq s\leq {m-1}$ are the column vectors of $A^m_\xi$
and $B^m_\xi=-\left(a_o^m,\ldots,a_{m-1}^m\right).$ Since $\det A^m_\xi=\prod\limits_{0\leq i<j\leq m-1}(a_j-a_i)\neq0,$
 the solution $X_\xi=(\tau_o,\ldots,\tau_{m-1})$ can be expressed as

\[\tau_k(\xi)=\begin{cases}
(-1)^m\mathop{\prod\limits_{i=0}^{m-1}}a_i & \text{if}~k=0,\\
(-1)^{m-1}\mathop{\sum\limits_{i=0}^{m-1}\prod\limits_{j\neq i}}a_j & \text{if}~k=1,\\
\hspace{.3in} \vdots &  ~\vdots\\
\sum\limits_{0\leq i<j\leq m-1}a_ia_j & \text{if}~k=m-2,\\
-\sum\limits_{i=0}^{m-1}a_i & \text{if}~k=m-1.
\end{cases}\]
Observe that $\tau_k$ are constant multiples of the elementary symmetric
polynomials in $a_j;~j\in F_{m-1}.$ Therefore, by Wiener's lemma, $\tau_k$ can
be locally agreed with the Fourier transform of some functions in $L^{1}(\mathbb{R}).$

\smallskip

Thus, the following sets in the projection $\Pi(\Lambda )$ will be dispensable in
the process of getting the determining sets $\Lambda$ for $X(\Gamma)$. Eventually,
dispensable sets are those sets in $\Pi (\Lambda)$ for which we could not solve (\ref{exp11}).

\smallskip

\noindent ${\bf{P_{1^\ast}}}.$ Since each $\xi\in\Pi^1(\Lambda)$ has unique
image in $\varSigma_{\xi},$ we can define a function $a_o$ on $\Pi^1(\Lambda)$
by $a_o(\xi)= e^{\pi i\eta_0}$, where $\eta_0=\eta_0(\xi)\in \varSigma_{\xi}.$
Now, the first dispensable set can be described by
\[\Pi^{1^\ast}(\Lambda)=\left\{\xi\in\Pi^1(\Lambda):a_o\in P^{1, p}[ L^{\Pi^1(\Lambda),\xi}_{loc}]\right\}.\]

Next, for $2\leq m\leq n+1$ we define the $m$-th dispensable set $\Pi^{m^\ast}(\Lambda)$
as follows. Given $\xi\in\Pi^m(\Lambda),$ let $a_j(\xi)= e^{\pi i\eta_j},$
where $\eta_j=\eta_j(\xi)\in \varSigma_{\xi};~j\in F_{m-1}.$
\smallskip

\noindent $\bf{P^{1}_{m^\ast}.}$
Since each $\xi\in\Pi^m(\Lambda)$ has $m$ distinct images in $\varSigma_{\xi},$
we define $m$ many functions $\tau_j$ on $\Pi^m(\Lambda);j\in F_{m-1}$ such that $X_\xi=(\tau_o,\ldots,\tau_{m-1})$
is the solution of (\ref{exp05}). Thus, the $m$-th dispensable set can be defined by
\[\Pi^{m^\ast}(\Lambda)=\left\{\xi\in\Pi^m(\Lambda):\tau_j\in L^{\Pi^2(\Lambda),\xi}_{loc}; ~j\in F_{m-1}\right\}.\]

Now, for $m\leq k\leq n+1,$  the corresponding dispensable sets can  be defined by
\smallskip

\noindent $\bf{P^{2}_{m^\ast}.}$ Further, we define $k$ many functions $\tau_j$ on
$\Pi^m(\Lambda);~j\in F_{k-1}$ such that $X_\xi=(\tau_o,\ldots,\tau_{k-1})$ satisfies
$A^{m,k}_\xi X_\xi=B^{m,k}_\xi,$ where $b_s=(a_o^s,\ldots,a_{m-1}^s);~0\leq s\leq {k-1}$
are the column vectors of $A^{m,k}_\xi$ and $B^{m,k}_\xi=-\left(a_o^k,\ldots,a_{m-1}^k\right).$
Hence, an auxiliary dispensable set can be described by
\[\Pi^k_{m^\ast}(\Lambda)=\left\{\xi\in\Pi^m(\Lambda):\tau_j\in L^{\Pi^m(\Lambda),\xi}_{loc}; ~j\in F_{k-1}\right\}.\]

\smallskip

\noindent $\bf{P^{3}_{m^\ast}.}$ Finally, we define $n+1$ many functions $\tau_j$
on $\Pi^m(\Lambda);~j\in F_n$ such that  $X_\xi=\left(\tau_o,\ldots,\tau_n\right)$
satisfies $A^{m,p}_\xi X_\xi=B^{m,p}_\xi,$ where $b_s=(a_o^s,\ldots,a_{m-1}^s);~0\leq s\leq n$
are the column vectors of $A^{m,p}_\xi$ and $B^{m,p}_\xi=-\left(a_o^p,\ldots,a_{m-1}^p\right).$
Thus, the $p$-th dispensable set can be defined by
\[\Pi^p_{m^\ast}(\Lambda)=\left\{\xi\in\Pi^m(\Lambda):\tau_j\in L^{\Pi^2(\Lambda),\xi}_{loc}; ~j\in F_n\right\}.\]
\smallskip

Now, we prove the following two results that speak a sharp contrast in the
pattern of dispensable sets as compared to dispensable sets corresponding
to two or three lines results. That is, a larger value of $p$ will
increase the size of dispensable sets in case of more than three finitely
many parallel lines.

\begin{lemma}\label{lemma99}
For $2\leq m\leq n,$ the following inclusion holds.
\[\Pi^{m^\ast}(\Lambda)\subset\Pi^{m+1}_{m^\ast}(\Lambda)\subset\cdots\subset\Pi^{n}_{m^\ast}(\Lambda)\subset\Pi^p_{m^\ast}(\Lambda).\]
\end{lemma}
\begin{proof}
If $\xi_o\in\Pi^{m^\ast}(\Lambda),$ then $\tau_j\in L^{\Pi^m(\Lambda),\xi_o}_{loc};~j\in F_{m-1}.$
Hence there exists an interval $I_{\xi_o}$ containing $\xi_o$ and  $\varphi_j\in L^1(\mathbb{R})$
such that $\tau_j=\hat{\varphi}_j;~j\in F_{m-1}$ and satisfy
\[\hat\varphi_o+\hat\varphi_1a_j+\cdots+\hat\varphi_{m-1}{a_j}^{m-1}+{a_j}^m=0\]
on $I_{\xi_o}\cap\Pi^m(\Lambda)$ for $j\in F_{m-1}.$ By similar iterations as
in the proof of Lemma \ref{lemma5}$(a),$ we infer that there exist a common set of $\psi_j\in L^1(\mathbb{R});~j\in F_m$
such that \[\hat\psi_o+\hat\psi_1a_j+\cdots+\hat\psi_m{a_j}^m+{a_j}^{m+1}=0\]
on $I_{\xi_o}\cap\Pi^m(\Lambda); ~j\in F_{m-1}.$ If we denote $\hat{\psi}_j=\tau_j,$
then $\xi_o\in\Pi^{m+1}_{m^\ast}(\Lambda).$ By induction it
follows that if $\xi_o\in\Pi^{k}_{m^\ast}(\Lambda),$ then $\xi_o\in\Pi^{k+1}_{m^\ast}(\Lambda)$
for $m+1\leq k\leq n.$ Since $\xi_o\in\Pi^{n+1}_{m^\ast}(\Lambda),$ there exist a common set of
functions $\psi_j\in L^1(\mathbb{R});~j\in F_n$ such that \[\hat\psi_o+\hat\psi_1a_j+\cdots+\hat\psi_n{a_j}^n+{a_j}^{n+1}=0\]
on $I_{\xi_o}\cap\Pi^m(\Lambda); ~j\in F_{m-1}.$ By similar iterations as in
the proof of Lemma \ref{lemma5}$(b),$ we infer that there exist a common set of
$\psi_j\in L^1(\mathbb{R});~j\in F_n$ such that \[\hat\psi_o+\hat\psi_1a_j+\cdots+\hat\psi_n{a_j}^n+{a_j}^p=0\]
on $I_{\xi_o}\cap\Pi^m(\Lambda); ~j\in F_{m-1}.$ If we denote $\hat{\psi}_j=\tau_j,$
then $\xi_o\in\Pi^p_{m^\ast}(\Lambda).$
\end{proof}

\begin{lemma}\label{lemma100}
The following inclusion holds.
\[\Pi^{{(n+1)}^\ast}(\Lambda)\subseteq\Pi^p_{{(n+1)}^\ast}(\Lambda).\] Moreover, equality
holds for $p=n+1.$
\end{lemma}
\begin{proof}
If $\xi_o\in\Pi^{(n+1)^\ast}(\Lambda),$ then $\tau_j\in L^{\Pi^{n+1}(\Lambda),\xi_o}_{loc}.$
Hence there exists an interval $I_{\xi_o}$ containing $\xi_o$ and  $\varphi_j\in L^1(\mathbb{R})$
such that $\tau_j=\hat{\varphi}_j;~j\in F_n$ satisfy
\[\hat{\varphi}_o+\hat{\varphi}_1a_j+\cdots+\hat{\varphi}_n{a_j}^n+{a_j}^{n+1}=0\]
on $I_{\xi_o}\cap\Pi^{n+1}(\Lambda)$ for $j\in F_n.$ By similar iterations as in the proof
of Lemma \ref{lemma5}$(b),$ we infer that there exists the common set of $\psi_j\in L^1(\mathbb{R});~j\in F_n$
such that \[\hat{\varphi}_o+\hat{\varphi}_1a_j+\cdots+\hat{\varphi}_n{a_j}^n+{a_j}^p=0\]
on $I_{\xi_o}\cap\Pi^{n+1}(\Lambda); ~j\in F_n.$ If we denote $\hat{\psi}_j=\tau_j,$
then it follows that $\xi_o\in\Pi^p_{(n+1)^\ast}(\Lambda).$
\end{proof}

Now, we state our main result of this section.

\begin{theorem}\label{th17}
Let $\Gamma=\mathbb R\times\left(F_n\cup\{p\}\right)$ and $\Lambda\subset\mathbb R^2$ be a
closed set which is $2$-periodic with respect to the second variable. Suppose
$\Pi(\Lambda)$ is dense in $\mathbb R.$ If $\left(\Gamma,\Lambda\right)$ is a
Heisenberg uniqueness pair, then the set
\[\widetilde\Pi(\Lambda)= {\Pi^{n+2}(\Lambda)}\bigcup\limits_{j=0}^n\left[{{\Pi^{(n-j+1)}(\Lambda)}\smallsetminus{\Pi^{{(n-j+1)}^\ast}(\Lambda)}}\right]\]
is dense in $\mathbb R.$
Conversely, if the set
\[\widetilde\Pi_p(\Lambda)= {\Pi^{n+2}(\Lambda)}\bigcup\limits_{j=0}^{n-1}\left[{{\Pi^{(n-j+1)}(\Lambda)}\smallsetminus{\Pi^p_{{(n-j+1)}^\ast}(\Lambda)}}\right]\cup\left[{{\Pi^{1}(\Lambda)}\smallsetminus{\Pi^{1^\ast}(\Lambda)}}\right]\]
is dense in $\mathbb R,$ then $\left(\Gamma,\Lambda\right)$ is a Heisenberg uniqueness pair.
\end{theorem}

\begin{remark}\label{rk7}
If $\Gamma$ is a system of $n+2$ consecutive parallel lines, then we get $p=n+1$ and from
Lemma \ref{lemma100} it follows that $\Pi^{(n+1)^\ast}(\Lambda)=\Pi^{n+1}_{(n+1)^\ast}(\Lambda).$
Thus, the necessary and sufficient conditions in Theorem \ref{th17} will be nearly identical.
However, In view of Lemma \ref{lemma99}, it follows that $\Pi^{m^\ast}(\Lambda)$ is properly
contained in  $\Pi^{n+1}_{m^\ast}(\Lambda);~2\leq m\leq n.$ Hence, $\widetilde\Pi_{n+1}(\Lambda)$
is a proper subset of $\widetilde\Pi(\Lambda).$ Thus, a unique necessary and sufficient condition
corresponding to finitely many lines is still open.
\end{remark}

We need the following two lemmas to prove the necessary part of Theorem~\ref{th17}.
The main purpose of these lemmas is to pull down an interval from some of the
partitioning sets of the projection $\Pi(\Lambda)$. The above argument helps
to negate the assumption that $\widetilde{\Pi }(\Lambda )$ is not dense in
$\mathbb{R}$. Now, we prove the following lemmas for $2\leq m\leq n+1.$

\begin{lemma}\label{lemma6}
Suppose $I$ is an interval such that $I\cap\Pi^{m^\ast}(\Lambda)$ is dense in $I.$
Then there exists an interval $I'\subset I$ such that
$I'\subset\Pi^{m^\ast}(\Lambda)\bigcup\limits_{j=m+1}^{n+2}\Pi^j(\Lambda).$
\end{lemma}

\begin{proof}
If $\xi_o\in I\cap\Pi^{m^\ast}(\Lambda),$ then there exist
$\tau_j\in L^{\Pi^m(\Lambda),\xi_o}_{loc};~j\in F_{m-1}$ such that
$X_{\xi_o}=(\tau_0,\ldots,\tau_{m-1})$ is the solution of (\ref{exp05}).
For $\xi\in\Pi^m(\Lambda),$ consider the matrix
\[\mathcal M^{m,n+1}_\xi=[b_o~\ldots~b_{m-2}~b_{n+1}],\]
where $b_s=(a_o^s,\ldots,a_{m-1}^s);~s\in F_{m-2}\cup\{n+1\}$ are the column vectors.
Now, we define a function $\rho$ on $\Pi^m(\Lambda)$ by
$\rho(\xi)=\left(\text{det}~\mathcal M^{m,n+1}_\xi\right)^2.$
\smallskip

Since $\tau_j;~j\in F_{m-1}$ are constant multiples of the  elementary symmetric
polynomials, by the fundamental theorem of symmetric polynomials,
$\rho$ can be expressed as a polynomial in $\tau_j;~j\in F_{m-1}.$ Moreover,
$\rho(\xi_o)\neq0.$ Hence it follows that $\rho\in L^{\Pi^m(\Lambda),\xi_o}_{loc}.$
By hypothesis, $I\cap\Pi^{m^\ast}(\Lambda)$ is dense in $I,$ there exists an interval
$I_{\xi_o}\subset I$ containing $\xi_o$ such that $\rho$ can be continuously extended
on $I_{\xi_o}.$ Thus, by continuity of $\rho$ on $I_{\xi_o},$ there exists an interval
$J\subset I_{\xi_o}$ containing $\xi_o$ such that $\rho(\xi)\neq0$ for all $\xi\in J.$

\smallskip

Consequently, $J\cap\Pi^{m^\ast}(\Lambda)$ is dense in $J$ and hence for $\xi\in J,$
there exists a sequence $\xi_n\in J\cap\Pi^{m^\ast}(\Lambda)$ such that $\xi_n\rightarrow\xi.$
Thus, the corresponding image sequences $\eta_j^{(n)}\in\varSigma_{\xi_n}\subseteq [0, 2)$
will have convergent subsequences, say $\eta_j^{(n_k)}$ which converge to $\eta_j;
~j\in F_{m-1}.$ Since the set $\Lambda$ is closed, $(\xi, \eta_j)\in\Lambda$ for
$j\in F_{m-1}.$

\smallskip

Next, we claim that all of $\eta_j;~j\in F_{m-1}$ are distinct. On the contrary,
suppose $\eta_j=\eta_k$ for some $j\neq k;~j,k\in F_{m-1}.$
Then by continuity of $\rho$ on $J,$ it follows that $\rho(\xi)=0,$ which contradicts
the fact that $\rho(\xi)\neq0$ for all $\xi\in J.$ Hence we infer that $J\subset\bigcup\limits_{j=m}^{n+2}\Pi^j(\Lambda).$
Further, using the facts that $\tau_j\in L^{\Pi^m(\Lambda),\xi_o}_{loc}$ and
 $J\cap\Pi^{m^\ast}(\Lambda)$ is dense in $J,$ the functions $\tau_j$ can be extended
continuously on an interval $I'\subset J$ containing $\xi_o$ such that $\tau_j(\xi)\neq0$
for all $\xi\in I'.$ That is, if $\xi\in I'\cap\Pi^m(\Lambda),$ then $\tau_j\in L^{\Pi^m(\Lambda),\xi}_{loc}$
and hence $\xi\in\Pi^{m^\ast}(\Lambda).$ Thus, we conclude that
$I'\subset\Pi^{m^\ast}(\Lambda)\bigcup\limits_{j=m+1}^{n+2}\Pi^j(\Lambda).$
\end{proof}

\smallskip

\begin{lemma}\label{lemma17}
Suppose $I$ be an interval such that $I\cap\Pi^p_{{(n+1)}^\ast}(\Lambda)$
is dense in $I.$ Then there exists an interval $I'\subset I$ such that
$I'\subset\bigcup\limits_{j=2}^{n+2}\Pi^j(\Lambda).$
\end{lemma}

\begin{proof}
If $\xi_o\in I\cap\Pi^p_{{(n+1)}^\ast}(\Lambda),$ then there exist
$\tau_j\in L^{\Pi^{n+1}(\Lambda),\xi_o}_{loc};~j\in F_n$
such that $X_{\xi_o}=\left(\tau_o,\ldots,\tau_n\right)$ is the solution of
$A^{n+1}_{\xi_o} X_{\xi_o}=B^{n+1,p}_{\xi_o}.$ Let $a_j=e^{\pi i\eta_j};~j\in F_n.$
Then by a simple calculation, it follows that $\tau_n$ satisfies the estimate
\[|\tau_n(\xi_o)|=|h_{p-n}(a_o,\ldots,a_n)|<C(p,n),\]
whenever $\xi_o\in I\cap\Pi^p_{{(n+1)}^\ast}(\Lambda).$ Since
$I\cap\Pi^p_{{(n+1)}^\ast}(\Lambda)$ is dense in $I,$ $\tau_n$ can be extended continuously
on a small neighborhood containing $\xi_o.$ Thus, there exists an interval $I'\subset I$
containing ${\xi_o}$ such that
\begin{equation}\label{exp25}
|\tau_n(\xi)|<C(p,n)
\end{equation}
for all $\xi\in I'.$
\smallskip

Consequently, $I'\cap\Pi^p_{{(n+1)}^\ast}(\Lambda)$ is dense in $I'$ and hence for
$\xi\in I',$ there exists a sequence $\xi_n\in I'\cap\Pi^p_{{(n+1)}^\ast}(\Lambda)$
such that $\xi_n\rightarrow\xi.$ Thus, the corresponding image sequences
$\eta_j^{(n)}\in\varSigma_{\xi_n}\subseteq [0, 2)$ have a convergent
subsequence, say $\eta_j^{(n_k)}$ which converges to $\eta_j;~j\in F_{n}$.
Since the set $\Lambda$ is closed, we get $(\xi, \eta_j)\in\Lambda;~j\in F_{n}.$

\smallskip

We show that all of $\eta_j;~j\in F_{n}$ cannot be identical.
If not, write $a_j^{(k)}=e^{\pi i\eta_j^{(n_k)}};~j\in F_{n},$ then
\[|\tau_n(\xi_{n_k})|=|h_{p-n}(a_o^{(k)},\ldots,a_n^{(k)})|\rightarrow C(p,n).\]
Further, by the continuity of $\tau_n$ on $I'$ it follows that
$|\tau_n(\xi_n)|\rightarrow|\tau_n(\xi)|.$ That is, $|\tau_n(\xi)|=C(p,n)$
which contradicts (\ref{exp25}). Thus, $ I'\subset\bigcup\limits_{j=2}^{n+2}\Pi^j(\Lambda).$
\end{proof}

\section{Proof of Theorem~\ref{th17}}

\begin{proof}[Proof of Theorem~\ref{th17}]
We first prove the sufficient part of Theorem \ref{th17}. Suppose the
set $\widetilde\Pi_p(\Lambda)$ is dense in $\mathbb R.$ Then we show that
$(\Gamma, \Lambda)$ is a Heisenberg uniqueness pair. We claim that
$\hat{f_k}\vert_{\widetilde\Pi_p(\Lambda)}=0$ for all $k\in F_{n+1}.$
Since $\hat{f_k}$ is continuous function which vanishing on a dense set
$\widetilde\Pi_p(\Lambda),$ it follows that $\hat{f_k}\equiv0$ for all
$k\in F_{n+1}.$ Thus $\mu=0.$

\smallskip

As $\Pi(\Lambda)$ decomposed into $n+2$ disjoint sets, the proof of
the above assertion will be carried out in the following $n+2$ steps.

\smallskip

\noindent $\bf{(S_{n+2}).} $ If $\xi\in\Pi^{n+2}(\Lambda),$ then there exist at
least $n+2$ distinct $\eta_j\in\varSigma_\xi$ for $j\in F_{n+1}.$  Since $\hat\mu$
is vanishes on $\Lambda,~\hat\mu(\xi,\eta_j)=0$ for all $j\in F_{n+1}.$ Hence
$\hat{f_k}(\xi)$ for $k\in F_{n+1}$ satisfies a homogeneous system of $n+2$ equations.
Since $\xi\in\Pi^{n+2}(\Lambda)$ and  $h_{p-n}(a_0,\ldots,a_{n-1},a_n)\neq h_{p-n}(a_0,\ldots,a_{n-1},a_{n+1}),$
it follows that $\hat{f_k}(\xi)=0$ for all $k\in F_{n+1}.$

\smallskip

Next, for $2\leq m\leq n+1,$ we have the following cases.
\smallskip

\noindent $\bf{(S_m).} $ If $\xi\in\Pi^m(\Lambda),$ then there exist $m$
distinct $\eta_j\in\varSigma_\xi$ such that $\hat\mu(\xi,\eta_j)=0;$ $j\in F_{m-1}.$
That is,
\begin{equation}\label{exp09}
\hat{f_0}(\xi)+\cdots+a_j^n(\xi)\hat f_n(\xi)+a_j^p(\xi)\hat f_{n+1}(\xi)=0,
\end{equation}
where $a_j(\xi)=e^{\pi i\eta_j};~j\in F_{m-1}.$ In case if $\hat f_{n+1}(\xi)\neq0,$
then by applying the Wiener's lemma to (\ref{exp09}), we conclude that
$\xi\in\Pi^p_{m^\ast}(\Lambda).$ Thus, for
$\xi\in{\Pi^m(\Lambda)\smallsetminus\Pi^p_{m^\ast}(\Lambda)},$
we have $\hat f_{n+1}(\xi)=0.$

\smallskip

Now, we continue the above argument for $m\leq k\leq n.$

\smallskip

Suppose $\hat f_k(\xi)\neq0$ and $\hat f_s(\xi)=0;~s=k+1,\ldots,n+1,$ then by applying
the Wiener's lemma to (\ref{exp09}), we get $\xi\in\Pi^k_{m^\ast}(\Lambda).$ Thus, for
$\xi\in{\Pi^m(\Lambda)\smallsetminus\Pi^k_{m^\ast}(\Lambda)},$ we infer that $\hat{f_k}(\xi)=0.$

\smallskip

Hence for $\xi\in{\Pi^m(\Lambda)\smallsetminus\Pi^p_{m^\ast}(\Lambda)},$
we infer that $\hat{f_k}(\xi)=0$ for all $k\in F_{n+1}.$

\smallskip

\noindent $\bf{(S_1).} $ If $\xi\in\Pi^1(\Lambda),$ then there exists unique
$\eta_0\in\varSigma_\xi$ such that $\hat\mu(\xi,\eta_0)=0.$
That is,
\begin{equation}\label{exp9}
\hat{f_0}(\xi)+\cdots+a_o^n(\xi)\hat f_n(\xi)+a_o^p(\xi)\hat f_{n+1}(\xi)=0,
\end{equation}
where $a_o(\xi)=e^{\pi i\eta_0}.$ In case if $\hat f_{n+1}(\xi)\neq0,$
then by applying Wiener's lemma to (\ref{exp9}), we conclude that $a_o\in P^{1,p}[L^{\Pi^1(\Lambda),\xi}_{loc}].$
Thus for $\xi\in{\Pi^1(\Lambda)\smallsetminus\Pi^{1^\ast}(\Lambda)},$
we have $\hat f_{n+1}(\xi)=0.$
\smallskip

We repeat the above argument for $1\leq k\leq n.$
\smallskip

Suppose $\hat f_k(\xi)\neq0$ and $\hat f_s(\xi)=0;~s=k+1,\ldots,n+1,$
then by applying the Wiener's lemma to (\ref{exp9}), we conclude that
$a_o\in P^{1,k}[L^{\Pi^1(\Lambda),\xi}_{loc}].$ Hence, in view of
Lemma \ref{lemma5}, it follows that $\xi\in\Pi^{1^\ast}(\Lambda).$
Thus, for $\xi\in{\Pi^1(\Lambda)\smallsetminus\Pi^{1^\ast}(\Lambda)},$
we infer that $\hat{f_k}(\xi)=0$ for all $k\in F_{n+1}.$ This completes
the proof of sufficient condition.
\end{proof}
\smallskip

Now, we prove the necessary part of Theorem \ref{th17}. Suppose $\left(\Gamma, \Lambda\right)$
is a Heisenberg uniqueness pair. Then we claim that the set $\widetilde\Pi(\Lambda)$
is dense in $\mathbb R.$ We observe that this is possible if the dispensable sets
$\Pi^{j^\ast}(\Lambda);~j\in F_{n+1}\smallsetminus\{0\}$ interlace to each other,
though these sets are disjoint among themselves.

\smallskip

If possible, suppose $\widetilde\Pi(\Lambda)$ is not dense in $\mathbb R.$ Then there
exists an open interval $I_o\subset\mathbb R$ such that $I_o\cap\widetilde\Pi(\Lambda)$
is empty. This, in turn, implies that
\begin{equation}\label{exp8}
\Pi(\Lambda)\cap I_o=\left(\bigcup_{j=1}^{n+1}\Pi^{j^\ast}(\Lambda)\right)\cap I_o.
\end{equation}
From (\ref{exp8}), we conclude that $I_o$ intersects only  the dispensable sets.
Now, the remaining part of the proof of Theorem \ref{th17} is a consequence of the
following two lemmas which provide the interlacing property of the dispensable sets
$\Pi^{j^\ast}(\Lambda);~j\in F_{n+1}\smallsetminus\{0\}.$

\begin{lemma}\label{lemma8}
There does not exist any interval $J\subset I_o$ such that
$\Pi(\Lambda)\cap J$ is contained in $\Pi^{j^\ast}(\Lambda);~j\in F_{n+1}\smallsetminus\{0\}.$
\end{lemma}
\begin{proof}
On the contrary, suppose there exists an interval $J\subset I_o$ such that
$\Pi(\Lambda)\cap J\subset\Pi^{j^\ast}(\Lambda)$ for some $j\in F_{n+1}\smallsetminus\{0\}.$ Since
$\Pi^{j^\ast}(\Lambda);~j\in F_{n+1}\smallsetminus\{0\}$ are disjoint, there could be following $n+1$
possibilities.
 \smallskip

\noindent $\bf{(a).}$ If $\xi_o\in\Pi(\Lambda)\cap J\subset\Pi^{1^\ast}(\Lambda),$
then $a_0\in P^{1,p}[L^{\Pi^1(\Lambda),\xi_o}_{loc}]$ and hence there exists
an interval $I_{\xi_o}\subset J$ containing $\xi_o$ and $\varphi_k\in L^1(\mathbb R);~k\in F_n$ such that
\[a_o^p+\hat{\varphi}_na_o^n+\cdots+\hat{\varphi}_1{a_o}+\hat{\varphi}_o=0\]
on $I_{\xi_o}\cap\Pi^1(\Lambda).$ Now, consider a function $f_{n+1}\in L^1(\mathbb R)$
such that $\hat f_{n+1}(\xi_o)\neq0$ and $\text{supp }\hat f_{n+1}\subset I_{\xi_o}.$
Let $f_k=f_{n+1}\ast\varphi_k;~k\in F_n,$ then we can construct a non-zero finite Borel
measure $\mu$ which is supported on $\Gamma$ that satisfies
\[\hat\mu(\xi,\eta)=\hat{f_0}(\xi) +
a_0(\xi)\hat{f_1}(\xi) +\cdots+a_0^n(\xi)\hat f_n(\xi) +a_0^p(\xi)\hat f_{n+1}(\xi)=0\]
for all $\xi\in I_{\xi_o}\cap\Pi^{1^\ast}(\Lambda),$ where $\eta\in\varSigma_\xi.$
In addition, from (\ref{exp8}) we conclude that $I_\xi\cap\Pi(\Lambda)=I_\xi\cap\Pi^{1^\ast}(\Lambda).$
This, in turn, implies $\hat\mu\vert_\Lambda=0.$
However, $\mu$ is a non-zero measure, which contradicts the fact that $\left(\Gamma, \Lambda\right)$ is a HUP.

\smallskip
Next, for $2\leq m\leq n$ we have the following cases.

\smallskip

\noindent $\bf{(b).}$  If $\xi_o\in\Pi(\Lambda)\cap J\subset\Pi^{m^\ast}(\Lambda),$ then
by Lemma \ref{lemma99}, $\xi_o\in\Pi^p_{m^\ast}(\Lambda)$ and hence there exists an  interval
$I_{\xi_o}\subset J$ containing $\xi_o$ and $\varphi_k\in L^1(\mathbb R);~k\in F_n$ such that
\[{a_j}^p+\hat{\varphi}_n{a_j}^n+\cdots+\hat{\varphi}_1{a_j}+\hat{\varphi}_o=0\]
on $I_{\xi_o}\cap\Pi^m(\Lambda)$ for $j\in F_{m-1}.$ Let $f_{n+1}\in L^1(\mathbb R)$ be such that
$\hat f_{n+1}(\xi_o)\neq0$ and $\text{supp }\hat f_{n+1}\subset I_{\xi_o}.$ Consider
$f_k=f_{n+1}\ast\varphi_k;~k\in F_n.$ Then we can construct a non-zero finite Borel measure $\mu$
which satisfies \[\hat\mu(\xi,\eta_j)=\hat{f_0}(\xi) +
a_j(\xi)\hat{f_1}(\xi) +\cdots+a_j^n(\xi)\hat f_n(\xi) +a_j^p(\xi)\hat f_{n+1}(\xi)=0\]
for all $\xi\in I_{\xi_o}\cap\Pi^{m^\ast}(\Lambda),$ where $\eta_j\in\varSigma_{\xi};~j\in F_{m-1}.$
Since $I_\xi\cap\Pi(\Lambda)=I_\xi\cap\Pi^{m^\ast}(\Lambda),$ it follows that
$\hat\mu\vert_\Lambda=0.$ This contradicts the fact that $\left(\Gamma, \Lambda\right)$ is a HUP.

\smallskip

\noindent $\bf{(c).} $ If $\xi_o\in\Pi(\Lambda)\cap J\subset\Pi^{(n+1)^\ast}(\Lambda),$ then
by Lemma \ref{lemma100}, $\xi_o\in\Pi^p_{(n+1)^\ast}(\Lambda)$ and hence there exists an  interval
$I_{\xi_o}\subset J$ containing $\xi_o$ and $\varphi_k\in L^1(\mathbb R);~k\in F_n$ such that
\[{a_j}^p+\hat{\varphi}_n{a_j}^n+\cdots+\hat{\varphi}_1{a_j}+\hat{\varphi}_o=0\]
on $I_{\xi_o}\cap\Pi^{n+1}(\Lambda)$ for $j\in F_n.$ Let $f_{n+1}\in L^1(\mathbb R)$ be such that
$\hat f_{n+1}(\xi_o)\neq0$ and $\text{supp }\hat f_{n+1}\subset I_{\xi_o}.$ Let
$f_k=f_{n+1}\ast\varphi_k;~k\in F_n.$ Then we can construct a non-zero finite Borel measure $\mu$ which satisfies
\[\hat\mu(\xi,\eta_j)=\hat{f_0}(\xi) +
a_j(\xi)\hat{f_1}(\xi) +\cdots+a_j^n(\xi)\hat f_n(\xi) +a_j^p(\xi)\hat f_{n+1}(\xi)=0\]
for all $\xi\in I_{\xi_o}\cap\Pi^{(n+1)^\ast}(\Lambda),$ where $\eta_j\in\varSigma_{\xi};~j\in F_n.$
Since $I_\xi\cap\Pi(\Lambda)=I_\xi\cap\Pi^{(n+1)^\ast}(\Lambda),$ it follows that
$\hat\mu\vert_\Lambda=0.$ This contradicts the fact that $\left(\Gamma, \Lambda\right)$ is a HUP.

\end{proof}

Next, we prove that there does not exist an interval $J\subset I_o$
which intersects more than one dispensable sets. Let $2\leq s\leq n+1.$
\begin{lemma}\label{lemma9}
There does not exist any interval $J\subset I_o$ such that $\Pi(\Lambda)\cap J$
is contained in $\bigcup\limits_{k=1}^s\Pi^{{n_k}^\ast}(\Lambda);~{n_k}\in F_{n+1}\smallsetminus\{0\}$
with $n_1<\cdots<n_s.$
\end{lemma}

\begin{proof}
$(i).$ For $s=2,$ suppose there exists an interval $J\subset I_o$ such that
$\Pi(\Lambda)\cap J\subset\Pi^{{n_1}^\ast}(\Lambda)\cup\Pi^{{n_2}^\ast}(\Lambda).$
Then, from (\ref{exp8}) we get
\begin{equation}\label{exp18}
J\cap\Pi(\Lambda)=J\cap\left(\Pi^{{n_1}^\ast}(\Lambda)\cup\Pi^{{n_2}^\ast}(\Lambda)\right).
\end{equation}
We claim that $J\cap\Pi^{{n_2}^\ast}(\Lambda)$ is dense in $J.$ If not,
suppose there exists an interval $I\subset J$ such that
$\Pi^{{n_2}^\ast}(\Lambda)\cap I=\varnothing.$ Then from (\ref{exp8}), we have
$I\cap\Pi(\Lambda)=I\cap\Pi^{{n_1}^\ast}(\Lambda)\subset\Pi^{{n_1}^\ast}(\Lambda)$
which contradicts Lemma \ref{lemma8}. Hence, by Lemma \ref{lemma6}, there exists
an interval $I'\subset J$ such that
$I'\subset\Pi^{{n_2}^\ast}(\Lambda)\bigcup\limits_{j={n_2}+1}^{n+2}\Pi^j(\Lambda)$
which contradicts that $I_o$ intersects only  the dispensable sets.
\smallskip

$(ii).$ Let $2<s\leq n+1.$ Then for $2\leq m<s,$ we assume that there does not
exist any interval $J\subset I_o$ such that
\begin{equation}\label{exp78}
\Pi(\Lambda)\cap J\subset\bigcup\limits_{k=1}^{m}\Pi^{{n_k}^\ast}(\Lambda);~{n_k}\in F_{n+1}\smallsetminus\{0\}.
\end{equation}
We claim that, there does not exist any interval $J\subset I_o$ such that
\begin{equation}\label{exp79}
\Pi(\Lambda)\cap J\subset\bigcup\limits_{k=1}^{m+1}\Pi^{{n_k}^\ast}(\Lambda);~{n_k}\in F_{n+1}\smallsetminus\{0\}.
\end{equation}
If possible, suppose there exists an interval $J\subset I_o$ such that
(\ref{exp79}) holds. Then $J\cap\Pi^{n_{m+1}^\ast}(\Lambda)$ is dense in $J.$
On contrary, suppose there exists an interval $I\subset J$ such that
$\Pi^{n_{m+1}^\ast}(\Lambda)\cap I=\varnothing.$ Then by (\ref{exp8}), it follows that
$I\cap\Pi(\Lambda)\subset\bigcup\limits_{k=1}^m\Pi^{{n_k}^\ast}(\Lambda),$
which in turn contradicts the assumption. Hence, by Lemma \ref{lemma6}, there exists
an interval $I'\subset J$ such that
$I'\subset\Pi^{n_{m+1}^\ast}(\Lambda)\bigcup\limits_{j=n_{m+1}+1}^{n+2}\Pi^j(\Lambda)$
which contradict the assumption that $I_o$ intersects only the dispensable sets.
This completes the proof of Theorem \ref{th17}.
\end{proof}

\smallskip

\section{Appendix A. Construction of partitions of $\Pi(\Lambda)$}
In this section, we elaborate the use of symmetric homogeneous polynomials
in reducing the following matrix to an upper triangular matrix that has crucial
role in considering the $p$-th line at large.

Consider the homogeneous system of linear equations $AX=0,$ where
entries of $A$ are satisfying
\[ a_{ij}=\begin{cases}
        \beta_i^j  & \text{ if }~(i,j)\in F_{n+1}\times F_n,\\
         \beta_i^p  & \text{ if }~j=n+1
   \end{cases}
\]
with $|\beta_i|=1;~i\in F_{n+1}.$ Next, we reduce the matrix $A$ to an
upper triangular matrix by using the notion of the symmetric polynomial.

\smallskip

For each $k\in\mathbb Z_+,$ the complete homogeneous symmetric polynomial $h_k$
of degree $k$ is the sum of all monomials of degree $k$ which can be derived
from the generating function identity
\[\sum\limits_{k=0}^{\infty}h_k\left(x_1,\ldots, x_s\right)t^k=\prod\limits_{i=1}^{s}\frac{1}{1-x_i t}. \]
For more details, we refer to \cite{MIG}. Notice that
\begin{eqnarray*}
&&\prod\limits_{i=1}^{s}\left(\frac{1}{1-x_i t}\right)\left(\frac{1}{1-xt}\right)-\prod\limits_{i=1}^{s}\left(\frac{1}{1-x_i t}\right)\left(\frac{1}{1-yt}\right)\\
&=& t(x-y)\prod\limits_{i=1}^{s}\left(\frac{1}{1-x_i t}\right)\left(\frac{1}{1-xt}\right)\left(\frac{1}{1-yt}\right).
\end{eqnarray*}
By equating the coefficients of $t^k$ in both sides, we get the identity
\begin{equation}\label{exp5}
h_k\left(\bar x,x\right)-h_k\left(\bar x,y\right)=(x-y)h_{k-1}\left(\bar x,x,y\right),
\end{equation}
where $\bar x=\left(x_1,\ldots,x_s\right).$ By applying the chain of
alternative row operations:
\begin{align*}
R'_i &= R_i-R_0;~i\in F_{n+1}\smallsetminus\{0\},\\
R''_i &= R'_i-\left(\frac{{\beta}_i-{\beta}_0}{{\beta}_1-{\beta}_0}\right)R'_1;~i\in F_{n+1}\smallsetminus\{0,1\}
\end{align*}
and the identity (\ref{exp5}), we obtain an upper triangular matrix with $(n+1,n+1)^\text{th}$ entry
equals to
\[\left(h_{p-n}\left({\beta}_0,\ldots,{\beta}_{n-1}, {\beta}_n\right)-h_{p-n}\left({\beta}_0,\ldots, {\beta}_{n-1},{\beta}_{n+1} \right)\right)\prod\limits_{j=0}^{n-1}({\beta}_{n+1}-{\beta}_j).\]

\section{Appendix B. Concluding remarks}

\noindent $\textbf{(a).}$ We observed that for $\xi\in\Pi(\Lambda),$ if
$\text{card}\left(\varSigma_\xi\right)\geq p+1$ then $\xi\in\Pi^{n+2}(\Lambda).$
For this, given $c\in\mathbb C,$ notice that there exists at most $p-n$ different
images $\eta_o^{(k)}\in [0,2)$ such that $\eta_o^{(k)}\neq\eta_s;~s\in F_{n}\smallsetminus \{0\}$
with all of $\eta_s$ are distinct and satisfies \[h_{p-n}(a_0^{(k)},\ldots,a_{n-1},a_n)=c,\]
where $a_j=e^{\pi i\eta_j};~j\in F_{n}.$

\bigskip

In particular, if we consider $n+2$ consecutive lines $\Gamma=\mathbb R\times F_{n+1},$
then $p=n+1$ and condition (\ref{exp12}) reduces to $a_n\neq a_{n+1}.$ Hence, the projections
$\Pi^{n+j}(\Lambda);~j=1, 2$ regain to their simplest form. That is,
\begin{align*}
\Pi^{n+1}(\Lambda)&=\{\xi\in\Pi(\Lambda):~\text{card}\left(\varSigma_\xi\right)={n+1}\},\\
\Pi^{n+2}(\Lambda)&=\{\xi\in\Pi(\Lambda):~\text{card}\left(\varSigma_\xi\right)\geq{n+2}\}.
\end{align*}
Thus, we infer that for a system of consecutive parallel lines, Theorem \ref{th17} can be
re-stated and could has a nearly unique necessary and sufficient condition.

\bigskip

\noindent $\textbf{(b).}$ Let $\Gamma=\mathbb R\times F_{n+1}$ and
$\Lambda=\mathbb R\times\left\{\frac{2k}{n+1};~k\in F_n\right\}.$
Then for any $\xi\in\mathbb R,$ $\text{card}\left(\varSigma_\xi\right)={n+1}.$
Thus, $\xi\in\Pi^{n+1}(\Lambda)$ that in turn implies $\Pi^{(n+1)^\ast}(\Lambda)=\mathbb R.$
For this, let $\xi\in\Pi^{n+1}(\Lambda)$ and denote $a_k=e^{\pi i\eta_k},$ where $\eta_k=\frac{2k}{n+1};~k\in F_n.$
Since $a_k^{n+1}=1,$ by a simple calculation, it follows that $X_\xi=((-1)^{n+1},0,\ldots,0)$
is the solution of $A^{n+1}_\xi X_\xi=B^{n+1}_\xi.$ Thus, $\xi\in\Pi^{(n+1)^\ast}(\Lambda).$
By Theorem \ref{th17}, we conclude that $(\Gamma,\Lambda)$ is not a HUP.

\bigskip

\noindent $\textbf{(c).}$ We observe a phenomenon of interlacing of $n+1$ totally
disconnected disjoint dispensable sets ${\Pi^{{(n-j+1)}^\ast}(\Lambda)}:~j\in F_n$
which are essentially derived from zero sets of $n+2$ trigonometric polynomials.
If Lemma \ref{lemma17} could be modified such that there exists $I'\subset I$ with
$I'\subset\Pi_{(n+1)^\ast}^p(\Lambda)\cup\Pi^{(n+2)}(\Lambda),$ then the necessary
condition of Theorem \ref{th17} would be minimized. Hence a characterization of
$\Lambda$ for finitely many lines might be obtained that would be closed to three
lines result. However, an exact analogue of three lines result for a large number of
lines is still open.

\bigskip

\noindent $\textbf{(d).}$ Finally, if we consider countably many parallel lines, then
whether the projection $\Pi(\Lambda)$ would be still relevant after deleting the countably
many dispensable sets, seems to be a reasonable question. However, since the dispensable
sets are totally disconnected, one can think of analyzing countably many lines for HUP in
terms of the Hausdorff dimension of the dispensable sets. We leave this question open for
the time being.

\bigskip

\noindent{\bf Acknowledgements:}\\

The authors extend a special thanks to Mr. Hiranmoy Pal for his contribution to \textbf{Appendix A}
about the notion of symmetric polynomials. Finally, we gratefully acknowledge the support
provided by IIT Guwahati, Government of India.

\bigskip


\end{document}